\documentclass[12pt,a4paper]{article}
\usepackage[top=3cm, bottom=3cm, left=3cm, right=3cm]{geometry}
\usepackage[latin1]{inputenc}
\usepackage{amsmath}
\usepackage{amsthm}
\usepackage{amsfonts}
\usepackage{amssymb}
\usepackage{graphics}
\usepackage{float}
\usepackage[hang,flushmargin]{footmisc} 

\usepackage{amsmath,amsthm,amssymb,graphicx, multicol, array}
\usepackage{enumerate}
\usepackage{enumitem}
\usepackage{xcolor}

\usepackage[pagebackref,bookmarks,colorlinks,breaklinks]{hyperref}

\hypersetup{linkcolor=blue,citecolor=red,filecolor=blue,urlcolor=blue} 

\usepackage{tabto}

\usepackage{tikz}
\usepackage{pgfplots}

\newtheorem{thm}{Theorem}[section]
\newtheorem{lem}{Lemma}[section]
\newtheorem{conj}{Conjecture}[section]
\newtheorem{exa}{Example}[section]
\newtheorem{cor}{Corollary}[section]
\newtheorem{dfn}{Definition}[section]
\newtheorem{exe}{Exercise}[section]

\newtheorem{openp}{Open Problem}[section]

\newcommand{\Z}{\mathbb{Z}}
\newcommand{\Q}{\mathbb{Q}}

\newcommand{\F}{\mathbb{F}}

\usepackage{mathtools}

\usepackage{fancyhdr}
\newlength{\bibitemsep}
\newlength{\bibparskip}
\let\oldthebibliography\thebibliography
\renewcommand\thebibliography[1]{%
	\oldthebibliography{#1}%
	\setlength{\parskip}{\bibitemsep}%
	\setlength{\itemsep}{\bibparskip}%
}

\title{Note on the Distribution of the Traces of Frobenius}
\date{}
\author{N. A. Carella}

\begin{document}
	\maketitle

\begin{abstract}
		Let $E$ be a nonsingular elliptic curve over the rational numbers, and let $\tau^n=p^n+1-\#E(\F_{p^n})$. A result in the current literature claims that the normalized traces of Frobenius $\alpha_n=(\tau^n+\overline{\tau}^n)/2p^{n/2}$ are equidistributed on the interval $[-1,1]$ as $n\to \infty$. This short note has two goals: 1. Proposes a correction to this result on the distribution of the traces of Frobenius. 2. Provides a simple elementary proof of this result. More precisely, it is shown that the sequence $\{\alpha_n:n\geq1\}$ is dense but not equidistributed on the interval $[-1,1]$.
			\let\thefootnote\relax\footnote{ \today \date{} 
				}
		\end{abstract}
		
\section{Introduction and the Result} \label{S2222}
Let $E$ be a nonsingular elliptic curve over the rational numbers, and let $\tau^n=p^n+1-\#E(\F_{p^n })$. A result in the literature on the theory of algebraic curves claims that the normalized traces of Frobenius $\alpha_n=(\tau^n+\overline{\tau}^n)/2p^{n/2}$ are equidistributed on the interval $[-1,1]$ as $n\to \infty$. \\

This claim would identifies the sequences traces of Frobenius as one of the first collection of sequences of powers of real numbers known to be equidistributed on the interval $[-1,1]$ or any other interval.  In Diophantine analysis, it is known that almost all sequences $\{\theta^n\bmod 1:n\geq1\}$ generated by the powers of real numbers $\theta>0$
are equidistributed on the unit interval $[0,1]$, but there is no specific examples, see \cite{PS1964}, \cite{BM2012}, \cite{AT2004}, et alii, for more details. \\

However, this claim about the sequence of powers of a trace of Frobenius seems to be incorrect. Even the probability distribution function of the trace of Frobenius, which has the form
		
		\begin{equation}\label{eq2222.100}
			f(t)=\frac{1}{\pi}\frac{1}{\sqrt{1-t^2}} \ne \text{constant},
		\end{equation}
		contradicts this claim.\\

The technical materials seem to demonstrate that a sequence of powers of a trace of Frobenius are similar to a sequence of powers of a Salem number, that is, dense in the interval, but not equidistributed as proved in \cite{AT2004}, \cite{DR2008}, \cite{SD2016}, et cetera. In fact, the graphs for the distributions of powers of Salem numbers of degree 4, see Definition \ref{dfn9795.700}, and the graphs for the distributions of the traces of Frobenius are nearly the same, compare the graphs in \cite[Section 4]{AT2004}, and \cite[Figure 1]{AS2023}.  \\
 
The simple elementary proof is given in Lemma \ref{lem9797.130}. The proofs given in \cite[Corollary 2.4.2.]{AS2023}, \cite[Corollary 2.11.]{SA2016}, etc., seem to be incorrect.	\\

\section{Simple Elementary Proof and Correction}\label{S9797}
Write the roots of the $L$-polynomial $L(E,T)=p^nT^2-\alpha_nT+1$ associated to the elliptic curve in the forms		
\begin{equation} \label{eq9797.110}
 			\tau^n=p^{n/2}e^{i n \theta} \quad \quad \text{ and}\quad \quad 
			\overline{\tau}^n=p^{n/2}e^{-i n \theta},
		\end{equation} 
		where $\theta \in[0,\pi]$, $|\alpha_n|\leq 2p^{n/2}$ and the normalized value
		
		\begin{equation} \label{eq9797.120}
			\alpha_n=\frac{\tau^n+\overline{\tau}^n}{2p^{n/2}} =\cos  n \theta
		\end{equation}
		for $n \geq 1$.

\begin{lem} \label{lem9797.130}  Let $E$ be a nonsingular ordinary elliptic curve over the rational numbers, and let $\tau^n=p^n+1-\#E(\F_{p^n})$. Then, the normalized traces of Frobenius $\alpha_n=(\tau^n+\overline{\tau}^n)/2p^{n/2}$ are dense but not equidistributed on the interval $[-1,1]$.
\end{lem}
			
\begin{proof}[\textbf{Proof}] The angle $\theta\ne0$ is an irrational number, see \cite[p.\;165]{PS1964}, \cite[Lemma 8]{AS2009}, and the sequence of real numbers $\{\theta n \mod \pi :n\geq1\}$ is uniformly distributed on the interval $[0,\pi]$, see \cite[Example 2.1]{KN1974} for similar detail. Now, the 1-to-1 and continuous map
\begin{align}\label{eq9797.140}
	[0,\pi]\quad &\longrightarrow \quad [-1,1]\\
s\quad &\longrightarrow \quad  -\cos  s \nonumber
\end{align}				
is \textit{dense invariant} but not \textit{uniform invariant}. It maps a dense set onto a dense set of real numbers, but it does not map an equidistributed set onto an equidistributed set of real numbers since $-\cos  s\ne as+b$ is not a linear map. These properties follow from the derived probability density function

\begin{eqnarray}\label{eq9797.150}
f(t)&=&p(t)\left | \frac{d}{dt}\cos^{-1}t \right |	\\
&=&\frac{1}{\pi}\frac{1}{\sqrt{1-t^2}},	\nonumber
\end{eqnarray}			
where $p(s)=1/\pi$ is the probability density function of the uniform random variable $s\in [0,\pi]$, and $f(t)$ is the probability density function of the random variable $-\cos s=t\in [-1,1]$. Setting $t=\alpha_n$, it readily follows that the sequence $\{\alpha_n: n\geq 1\}$ is dense on the interval $[-1,1]$ but not equidistributed. 
\end{proof}

The first proof of Lemma \ref{lem9797.130} based on the Weyl's criterion seems to be that given in \cite[p.\;166]{PS1964}, and repeated in \cite{DR2008}, \cite{SD2016} and similar references. The last part of the proof shows that the sequence is not equidistributed in the following way. Let $u_n=\alpha_n=-\cos \theta n$ be the sequence of traces of Frobenius. Applying the Weyl's criterion, \cite[Theorem 2.1]{KN1974}, yields	

\begin{eqnarray}\label{eq9797.160} 
\lim_{x\to \infty} \frac{1}{x}\sum_{n\leq x}e^{i 2 \pi  ku_ n}
&=&\lim_{x\to \infty} \frac{1}{x}\sum_{n\leq x}e^{-i 2 \pi  k\cos \theta n}\\
&=&\int_0^1e^{-i 2 \pi  k\cos 2\pi t}\nonumber\\
&=&J_0\left(2\pi k \right)\nonumber\\
&\ne&0,\nonumber
\end{eqnarray} 
where $k\ne0$ is an integer and 
\begin{equation}\label{eq9797.170} 
J_0(z)=\frac{1}{2\pi}\frac{1}{z^{1/2}}+O\left( \frac{1}{z^{3/2}}\right) >0
\end{equation}
is a Bessel function, see \cite[Lemma 2.2.]{DR2008}. Therefore, the sequence of traces of Frobenius $\{\alpha_n:n\geq 1\}$ is not equidistributed on the interval $[-1,1]$.\\

The proofs given in \cite{DR2008}, \cite{SD2016}, etc, generalizes to curves of genus $g\geq1$ in a very nice way. Furthermore, the derivations of the probability density functions, which are discontinuous, seems to explain the spikes shown on the experimental graphs obtained in \cite{AS2023}, \cite{FS2014}, etc. Another completely different generalization of the same result appears in \cite{AS2009}.

\begin{exa}{\normalfont Let $E:y^2=x^3+x+1$ over the finite field $\F_{13}$. The trace of Frobenius is
		\begin{equation}\label{eq9797.180a} 
			\alpha_1=p+1-\#E\left(  \F_p	\right) =\sum_{1\leq n\leq 13}\left( \frac{n^3+n+1}{13}\right)=4, 
		\end{equation}		
		and the angle 
		\begin{equation}\label{eq9797.180b} 
			\theta_0=\arccos \left( \frac{4}{2\sqrt{13}}	\right)= 
			0.9827937232473290679857106110146660144\ldots.
		\end{equation}
		is irrational, \cite[Lemma 8]{AS2009}. The sequence of real numbers
		\begin{eqnarray}\label{eq9797.180c} 
			\alpha_n&=&\tau^n+\overline{\tau}^n\\
			&=&\frac{\sqrt{13^n}e^{i\theta_0n}+\sqrt{13^n}e^{-i\theta_0n}}{2\sqrt{13^n}}\nonumber\\
			&=&-\cos \theta_0 n	=\cos ( 
			0.9827937232473290679857106110146660144\ldots) n\nonumber
		\end{eqnarray}
		is dense but not equidistributed on the interval $[-1,1]$ as $n\to \infty$. Here, the number $\tau^n$ and its conjugate $\overline{\tau}^n$ are roots of the $L$-polynomial $L(E,T)=p^nT^2-\alpha_nT+1$.
	}
\end{exa}

\section{Summatory Functions of Traces of Frobenius}\label{S9790}
\begin{cor} \label{cor9797.200}For any integer $k\ne0$ the summatory function
	\begin{equation}\label{eq9797.175a} 
		\sum_{n\leq x}e^{-i 2 \pi  k\cos \theta n}=	J_0(2\pi k)x+o(x). 
	\end{equation}
\end{cor}
This result complements the related twisted exponential sum
\begin{equation}\label{eq9797.175b} 
	\sum_{n\leq x}\mu(n) \cos \theta n=o(x) ,
\end{equation}
where $\mu$ is the Mobius function, proved in \cite[Theorem 1.1]{SS2019}.\\

\section{Discrepancy of the Sequence of Traces }\label{S9793}
The discrepancy of a sequence of real numbers $\{x_n:n\geq1\}$ on the unit interval $[0,1]$ is defined by the inequality
\begin{equation}\label{eq9793.175a}
	D_N(x_1,x_2,\ldots,x_N)=\sup_{0\leq \alpha\leq 1}\left |\#\{n\leq N: x_n\in[0,\alpha]-\alpha N\}\right |.
\end{equation}
The discrepancy is essentially a measure of uniformity and distribution of the sequences on the unit interval $[0,1]$. The discrepancies of uniformly distributed and nearly uniformly distributed sequences are in the range 
\begin{equation}\label{eq9793.175b}
\frac{c_0}{N}\leq 	D_N(x_1,x_2,\ldots,x_N)\leq 
\frac{c_1}{f(N)},
\end{equation}
where $f(n)$ is an unbounded function and
$c_0,c_1>0$ are constants, see \cite [Corollary 1.1]{KN1974}. The upper bound of the discrepancy is computed via the following analytic tool. 

 \begin{thm} \label{thm9793.175} {\normalfont (Erdos-Turan inequality)} If $\{x_n:n\geq1\}$ on the unit interval $[0,1]$, then for any
		$$D_N\leq 5\left(\frac{1}{x+1}+\frac{1}{N}\sum_{k\leq x}\frac{1}{k}\bigg |\sum_{n\leq N}e^{i2\pi k x_n}\bigg| \right) $$
		for any large real numbers $N$ and $x$.
		
	\end{thm}

\begin{proof}[\textbf{Proof}] A complete proof appears in \cite[Theorem 2.5]{KN1974} or , 
\end{proof}
  
The discrepancy test is a limit test applied to the Erdos-Turan inequality.
\begin{thm} \label{thm9793.200} If $\{x_n:n\geq 1\}$ is a sequence of real numbers on the unit interval, the following statements are equivalent.
	\begin{enumerate} [font=\normalfont, label=(\roman*)]
	\item The sequence $\{x_n:n\geq 1\}$ is uniformly distributed on the unit interval $[0,1].$ 
	
	\item The limit $\displaystyle \lim_{N\to\infty}D_N(x_1,x_2,\ldots,x_N)=0.$ 
\end{enumerate}
\end{thm}
\begin{proof}[\textbf{Proof}] A long proof based on measure theory is given in \cite[Theorem 1.6]{DT1997} and a short proof based on the Erdos-Turan inequality is given in \cite[Theorem 4.1.14]{NW2015}.
\end{proof}

This section concludes with the calculation of the discrepancy of the sequence of traces and an open problem.
\begin{cor} \label{cor9793.200} The discrepancy of the sequence of normalized traces of Frobenius $\{\alpha_n:n\geq1\}$ satisfies the inequality
	\begin{equation}\label{eq9793.175ca} 
D_N(\alpha_1,\alpha_2,\ldots,\alpha_N)\leq  \frac{c_1}{\sqrt{N}},
\end{equation}
where $c_1>0$ is a constant.
\end{cor}

\begin{proof}[\textbf{Proof}] Using the asymptotic form of the Bessel function in \eqref{eq9797.170} and plugging the value given into \eqref{eq9797.160} yield
\begin{equation}\label{eq9793.200b} 
	J_0(2\pi k)=\frac{1}{2\pi}\frac{1}{(2\pi k)^{1/2}}+O\left( \frac{1}{(2\pi k)^{3/2}}\right) >0.
\end{equation}	Replacing this into the Erdos-Turan inequality yields
\begin{eqnarray}\label{eq9793.200c}
D_N(\alpha_1,\alpha_2,\ldots,\alpha_N)&\leq&5\left(\frac{1}{x+1}+\frac{1}{N}\sum_{k\leq x}\frac{1}{k}\bigg |\sum_{n\leq N}e^{i2\pi k \alpha_n}\bigg| \right)\\
	&\leq&	 5\left(\frac{1}{x+1}+\sum_{k\leq x}\frac{1}{k}\left( \frac{1}{2\pi}\frac{1}{(2\pi k)^{1/2}}+O\left( \frac{1}{(2\pi k)^{3/2}}\right)\right)\right)   \nonumber \\	
	&\leq& 5\left(\frac{1}{x+1}+\frac{1}{(2\pi)^{3/2}N^{1/2}}+O\left( \frac{1}{(2\pi N)^{3/2}}\right)\right) \nonumber.	
\end{eqnarray}
Setting $x=N$ completes the verification.
\end{proof}

The discrepancy of sequences traces of Frobenius, and other sequences of real numbers with similar properties as the sequences of powers of Salem numbers satisfy the limit condition
\begin{equation}\label{eq9793.215a}
	\lim_{N\to\infty}D_N(x_1,x_2,\ldots,x_N)=0.
\end{equation}
However, these sequences are dense but not uniformly distributed on the unit interval $[0,1]$. This observation 
prompts an important question about the proof of the discrepancy test described in Theorem \ref{thm9793.200}.\\

\begin{openp}\label{openp9793.200}{\normalfont If the proof of Theorem \ref{thm9793.200} is complete, it has no gap, then how can it be used to distinguish between 
\begin{enumerate}
\item A dense sequence of real numbers $\{u_n:n\geq 1\}.$
\item An uniformly distributed sequence of real numbers $\{u_n:n\geq 1\}.$\\

Is it possible that there is a threshold of the form
 	\item A dense sequence of real numbers $\{u_n:n\geq 1\}$ if and only if \begin{equation}\label{eq9793.220a}
 D_N(x_1,x_2,\ldots,x_N)=O\left( \frac{\sqrt{\log \log N}}{N^{1/2}}\right) ,
 	\end{equation}
 	see \cite[Theorem 1.93]{DT1997} for details on the upper bound in \eqref{eq9793.220a}.
 	\item An uniformly distributed sequence of real numbers $\{u_n:n\geq 1\}$ if and only if 
 \begin{equation}\label{eq9793.220b}D_N(x_1,x_2,\ldots,x_N)=O\left(  \frac{1}{N^{1-\varepsilon}}\right) , 	
 \end{equation}
 where $\varepsilon>0$ is a small number, this the discrespancy of algebraic irrational numbers, see \cite[Example 3.1]{KN1974} for detailed proof.
 \end{enumerate}
}
\end{openp}

\section{Salem Numbers and the Density Functions }\label{S9795}
The powers sequences of a few collections of polynomials have the same distributions as the sequences of traces of Frobenius. For example, the powers sequence generated by the roots of the shifted $5th$ cyclotomic polynomial $\Phi_5(T)-2=T^4+T^3+T^2+T-1$. The best known collection in the literature is the set of Salem polynomials.
\begin{dfn}\label{dfn9795.700}{\normalfont
A Salem number is an algebraic integer $\tau>1$ of degree $d=2e\geq4$ such that the conjugates $\alpha_1,\;\alpha_2,\;\ldots,\; \alpha_{d-1},$ have their moduli $|\alpha_i|\leq  1$.}
\end{dfn}

\begin{thm}\label{thm9795.700} {\normalfont (\cite{PS1964})} Let $\tau$ be a Salem number. Then,
	\begin{enumerate} [font=\normalfont, label=(\roman*)]
	\item The powers $\tau^n$ reduced modulo $1$ are everywhere dense in $[0,1]$ as $n\to\infty$.
	\item The powers $\tau^n$ are not uniformly distributed modulo $1$ in $[0,1]$ as $n\to\infty$.
\end{enumerate}
\end{thm}
The exact probability densities functions of the sequence of powers $\{\tau^n:n\geq1\}$ of a Salem numbers $\tau$ of degree $d=2e\geq4$ are derived in \cite{SD2016}. Here a simple approximation is described. This approximation is accomplished by the arcsin probability density function
\begin{equation}\label{eq9795.500}
f_d(z)=\frac{1}{2(d-1)\sin^{-1}(1/(d-1))}\frac{1}{\sqrt{1-(z/(d-1))^2}}
\end{equation}
where $t\in [-1,1]$. The sequence of approximations $\{f_d(z):d\geq 4\}$ rapidly converges to the uniform density function $f(z)=1/2$ over the interval $[-1,1] $ as $d\to \infty$. The graph for the probability density function
\begin{equation}\label{eq9795.500b}
f_{4}(z)=\frac{1}{\pi}\frac{1}{\sqrt{1-z^2}}
\end{equation}
associated with an irreducible polynomial $p_4(T)\in \Z[T]$ of degree $d=4$, is displayed in Figure \ref{f9795.100a}. Likewise, the graph for the probability density function
\begin{equation}\label{eq9795.500c}
	f_{12}(z)=\frac{1}{10\sin^{-1}(1/5)}\frac{1}{\sqrt{1-(z/5)^2}}
\end{equation}
associated with an irreducible polynomial $p_{12}(T)\in \Z[T]$ of degree $d=12$, is displayed in Figure \ref{f9795.100b}.

\begin{figure}[H]
\setlength{\tabcolsep}{0.5cm}
\renewcommand{\arraystretch}{1.250092}
\setlength{\arrayrulewidth}{0.75pt}
\begin{tikzpicture}
	\begin{axis}[	
		xmin = -1.5, xmax = 1.5,
		ymin = 0, ymax = 3.00,
		xtick distance = .5,
		ytick distance = .5,
		grid = both,
		samples=100,
		axis lines=left, xlabel=$z$, ylabel=$f_4(z)$,
		width = 1.0\textwidth,
		height = 0.4\textwidth, very thick,
		grid = major,tick]
\addplot [
domain=-2:2, 
samples=100, 
color=blue,
]
{1/(3.1415*(1-x^2)^.5)};
	\end{axis}
\end{tikzpicture}
\caption{Powers distribution for algebraic Salem number of degree $4$.}
\label{f9795.100a}
\end{figure}
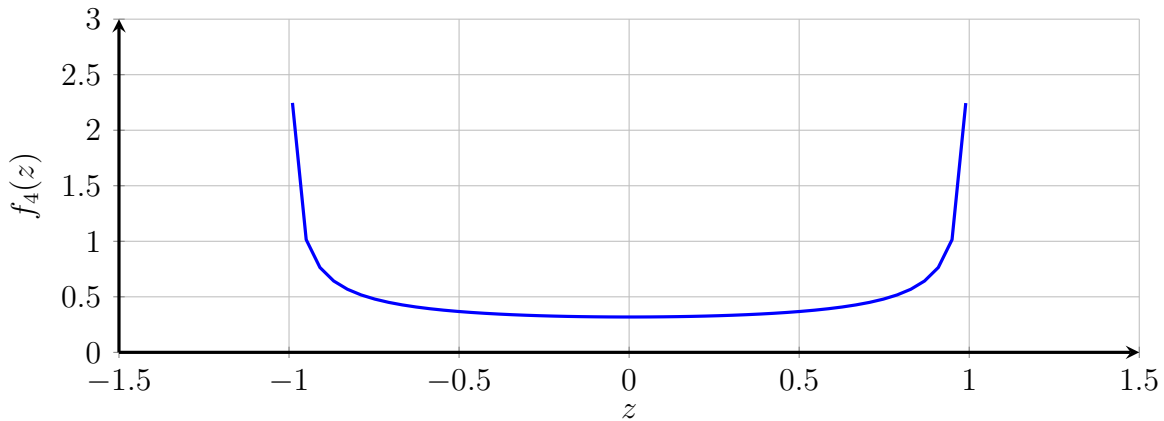

\begin{figure}[H]
	\setlength{\tabcolsep}{0.5cm}
	\renewcommand{\arraystretch}{1.250092}
	\setlength{\arrayrulewidth}{0.75pt}
	\begin{tikzpicture}
		\begin{axis}[	
			xmin = -1.5, xmax = 1.5,
			ymin = 0, ymax = 2.00,
			xtick distance = .5,
			ytick distance = .5,
			grid = both,
			samples=100,
			axis lines=left, xlabel=$z$, ylabel=$f_{12}(z)$,
			width = 1.0\textwidth,
			height = 0.4\textwidth, very thick,
			grid = major,tick]
			\addplot [
			domain=-1:1, 
			samples=100, 
			color=blue,
			]
			{1/(2.0135792079*(1-x^2/25)^.5)};
		\end{axis}
	\end{tikzpicture}
\caption{Powers distribution for algebraic Salem number of degree $12$.}
	\label{f9795.100b}
\end{figure}
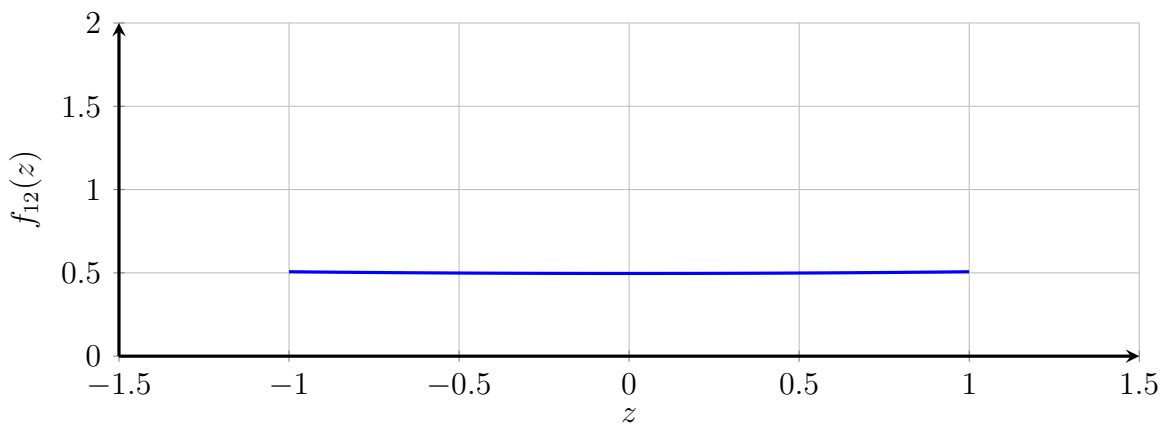

\section{Sato-Tate Distribution Function and Beyond}\label{S9799A}
The distribution of the angles of Frobenius as the primes varies follows a circular distribution known as the Sato-Tate distribution 
\begin{equation}\label{eq9799.200}
F(a,b)=\lim_{x\to\infty}\frac{\#\{p\leq x:a\leq \theta_p\leq b\}}{\pi(x)}=\frac{2}{\pi}\int_a^b \sin^2(\theta) d\theta,
\end{equation}
where $a_1=2p^{1/2}\cos \theta_p$ is the normalized trace and $[a,b]\subseteq [0,\pi]$.\\

The result for elliptic curves with complex multiplication was proved over half a century ago. The subset of traces $a_1=0$ forces a discontinuity on the probability density function. On the otherhand, the probability density function for elliptic curves without complex multiplication is continuous.

\begin{thm}\label{thm9799.300A} {\normalfont (Hecke)} Let $E$ be a nonsingular elliptic curve over the rational numbers $\Q$ with complex multiplication and conductor $N\geq1$. Assume $L(T)=T^2-a_1T+p$ is the characteristic polynomial and $\alpha_1=a_1/2p^{1/2}\in [-1,1]$ is the normalized trace. Then, for every interval $[a,b]\subseteq [-1,1]$ such that $0\notin [a,b]$ the cumulative distribution function is the following. 
\begin{enumerate} [font=\normalfont, label=(\roman*)]
	\item $\displaystyle F_E^0(a,b)=\lim_{x\to\infty}\frac{\#\{p\leq x:p\nmid N \text{ and }\alpha_1=0\}}{\pi(x)}=\frac{1}{2}.$ 
\item $\displaystyle F_E^r(a,b)=\lim_{x\to\infty}\frac{\#\{p\leq x:p\nmid N \text{ and }\alpha_1\in [a,b]\}}{\pi(x)}=\frac{1}{2\pi}\int_a^b\frac{1}{\sqrt{1-t^2}}dt$ .
\end{enumerate}	
\end{thm} 

\begin{cor}\label{cor9799.300B} Let $E$ be a nonsingular elliptic curve over the rational numbers $\Q$ with complex multiplication and conductor $N\geq1$. Assume $L(T)=T^2-a_1T+p$ is the characteristic polynomial and $\alpha_1=a_1/2p^{1/2}\in [-1,1]$ is the normalized trace. Then, for every interval $[a,b]\subset [-1,1]$ such that $0\notin [a,b]$ the primes counting function is the following.  
	\begin{enumerate} [font=\normalfont, label=(\roman*)]
		\item $\displaystyle \pi_E(a,b,x)=\#\{p\leq x:p\nmid N \text{ and }a_1=0\}=\frac{\pi(x)}{2}.$ 
		\item $\displaystyle \pi_E(a,b,x)=\#\{p\leq x:p\nmid N \text{ and }\alpha_1\in [a,b]\}=\frac{\pi(x)}{2\pi}\int_a^b\frac{1}{\sqrt{1-t^2}}dt$ .
	\end{enumerate}	
\end{cor}

But, the result for elliptic curves without complex multiplication is a new achievement, see \cite{AH2009} for some recent developments in this topic. 
\begin{thm}\label{thm9799.300C} {\normalfont (Cozel, et alii)} Let $E$ be a nonsingular elliptic curve over the rational numbers $\Q$ without complex multiplication and conductor $N\geq1$. Assume $L(T)=T^2-a_1T+p$ is the characteristic polynomial and $\alpha_1=a_1/2p^{1/2}\in [-1,1]$ is the normalized trace. Then, for every interval $[a,b]\subseteq [-1,1]$ the cumulative distribution function is the following. 
$$\displaystyle F_E(a,b)=\lim_{x\to\infty}\frac{\#\{p\leq x:p\nmid N \text{ and }\alpha_1\in [a,b]\}}{\pi(x)}=\frac{2}{\pi}\int_a^b\sqrt{1-t^2}dt.$$
\end{thm} 

\begin{cor}\label{cor9799.300D} Let $E$ be a nonsingular elliptic curve over the rational numbers $\Q$ with complex multiplication and conductor $N\geq1$. Assume $L(T)=T^2-a_1T+p$ is the characteristic polynomial and $\alpha_1=a_1/2p^{1/2}\in [-1,1]$ is the normalized trace. Then, for every interval $[a,b]\subseteq [-1,1]$ the primes counting function is the following.  
$$\displaystyle \pi_E(a,b,x)=\#\{p\leq x:p\nmid N \text{ and }\alpha_1\in [a,b]\}=\frac{2\pi(x)}{\pi}\int_a^b\sqrt{1-t^2}dt.$$
\end{cor}
Similar discussion regarding the cumulative density functions is given in \cite[p.\;283]{KZ2016}.\\

This section closes with the cumulative density functions of the normalized traces \eqref{eq9797.120}, these results are similar to the Hecke's results in Theorem \ref{thm9799.300A} and Corollary \ref{cor9799.300B}, but no proofs will be given. 
\begin{thm}\label{thm9799.400A} Let $E$ be a nonsingular elliptic curve over the rational numbers $\Q$ with complex multiplication and conductor $N\geq1$. Assume $p\nmid N$ is a prime of good reduction. Let $L(T)=T^2-a_nT+p^n$ is the characteristic polynomial and $\alpha_n=a_n/2p^{n/2}\in [-1,1]$ is the normalized trace. Then, for every interval $[a,b]\subseteq [-1,1]$ such that $0\notin [a,b]$ the cumulative distribution function is the following. 
	\begin{enumerate} [font=\normalfont, label=(\roman*)]
		\item $\displaystyle G_E^0(a,b)=\lim_{x\to\infty}\frac{\#\{n\leq x:\alpha_n=0\}}{x}=\frac{1}{2}.$ 
		\item $\displaystyle G_E^r(a,b)=\lim_{x\to\infty}\frac{\#\{n\leq x:\alpha_1\in [a,b]\}}{x}=\frac{1}{2\pi}\int_a^b\frac{1}{\sqrt{1-t^2}}dt$ .
	\end{enumerate}	
\end{thm}

\begin{cor}\label{cor9799.400B} Let $E$ be a nonsingular elliptic curve over the rational numbers $\Q$ with complex multiplication and conductor $N\geq1$. Assume $p\nmid N$ is a prime of good reduction. Let $L(T)=T^2-a_nT+p^n$ is the characteristic polynomial and $\alpha_n=a_n/2p^{n/2}\in [-1,1]$ is the normalized trace. Then, for every interval $[a,b]\subseteq [-1,1]$ such that $0\notin [a,b]$ the integers counting function is the following.  
	\begin{enumerate} [font=\normalfont, label=(\roman*)]
		\item $\displaystyle N_E(a,b,x)=\#\{n\leq x:a_n=0\}=\frac{x}{2}.$ 
		\item $\displaystyle N_E(a,b,x)=\#\{n\leq x:\alpha_n\in [a,b]\}=\frac{x}{2\pi}\int_a^b\frac{1}{\sqrt{1-t^2}}dt$ .
	\end{enumerate}	
\end{cor}

The corresponding results for noncomplex elliptic curves are presumably similar to the results by Cozel, et alli in Theorem \ref{thm9799.300C}.

\section{Lange-Trotter Conjecture}\label{S9799B}
A well known open problem gives a precise asymptotic for each fixed value of the trace of Frobenius.

As in the other distribution problems, the case for elliptic curves are broken into two classes. Elliptic curves with complex multiplication and elliptic curves without complex multiplication. 
\begin{conj}\label{conj9799.300} {\normalfont (Lange-Trotter)} Let $E$ be a nonsingular elliptic curve over the rational numbers $\Q$ without complex multiplication and conductor $N\geq1$. If $L(T)=T^2-a_1T+p$ is the characteristic polynomial with $a_1\ne0$, then 
	$$\pi_E^r(x)=\#\{p\leq x:p\nmid N \text{ and }a_1=r\}=c(E,r)\frac{x^{1/2}}{\log x}+O\left(\frac{x^{1/2}}{(\log x)^2} \right), $$
	where $|a_1|<2p^{1/2}$ and $c(E,r)>0$ is a constant.
	
\end{conj}

The derivation of the constant $c(E,r)$ is given in \cite[Equation 5]{CP1999}, et alii. Several new contributions to this topic have been completed, see \cite{AH2009} for some recent developments in this topic.

\newpage
\section{Problems}\label{exe9797}

\begin{exe}\label{exe9797.405a} {\normalfont
Let $f(T)\in\Z[T]$ be an irreducible polynomial of degree $\deg f=d$, and let $\alpha_1,	\alpha_2,\ldots,\alpha_d$ be its set of roots.	Show that the sum of root powers 
		$$\alpha_1^n+\alpha_2^n+\cdots+\alpha_d^n\in \Z$$	
is an integer for all integer $n\geq0$.
}
\end{exe}

\begin{exe}\label{exe9797.405b} {\normalfont
		Let $\Phi_{5}(T)$ be the 5th cyclotomic polynomial, let $f(T)=\Phi_{5}(T)-3\in\Z[T]$ be a polynomial of degree $\deg f=4$, and let $\alpha_1,	\alpha_2,\alpha_2\alpha_3,\alpha_{4}$ be its set of roots. The first two roots $\alpha=\alpha_1=-1.4469,	\alpha_2=0.74127\ldots$ are real, and the other two complex conjugates. Show that the powers sequence $\{\alpha^n:n\geq 1\}$ is dense but not equidistributed in $[0,1]$. Explain why this is not a Salem number?
	}
\end{exe}

\begin{exe}\label{exe9797.405c} {\normalfont
		Let $\Phi_{13}(T)$ be the 13th cyclotomic polynomial, let $f(T)=\Phi_{13}(T)-3\in\Z[T]$ be a polynomial of degree $\deg f=12$, and let $\alpha_1,	\alpha_2,\ldots,\alpha_{12}$ be its set of roots. The first two roots $\alpha=\alpha_1=-1.3878,	\alpha_2=0.66844\ldots$ are real, and the other are complex conjugates. Show that the powers sequence $\{\alpha^n:n\geq 1\}$ is dense but not equidistributed in $[0,1]$. Explain why this is not a Salem number?
	}
\end{exe}
\begin{exe}\label{exe9797.405} {\normalfont
Use the Bessel function integral representation $J_0(z)=\pi^{-1}\int_0^{\pi}e^{iz\cos w}dw$,  to convert a infinite sum to a Bessel function
$$\lim_{x\to \infty} \frac{1}{x}\sum_{n\leq x}e^{-i 2 \pi  k\cos \theta n}=\int_0^1e^{-i 2 \pi  k\cos 2\pi t}=J_0(2\pi k).$$	
	}
\end{exe}

\begin{exe}\label{exe9797.405d} {\normalfont
Let $z\in [-1,1]$. Show that the sequence arcsin probability density functions has the limit
		$$\lim_{d\to \infty} \frac{1}{2(d-1)\sin^{-1}(1/(d-1))}\frac{1}{\sqrt{1-(z/(d-1))^2}}=\frac{1}{2}.$$	
	}
\end{exe}

%


\end{document}